\documentclass[3p,11pt]{elsarticle}

\usepackage{amsthm,amsmath,amssymb}

\usepackage{xspace}

\numberwithin{equation}{section}

\newtheorem{theorem}{Theorem}[section]
\newtheorem{proposition}[theorem]{Proposition}
\newtheorem{lemma}[theorem]{Lemma}

\newtheorem{corollary}[theorem]{Corollary}

\newcommand{\secref}[1]{Section~\ref{#1}\xspace}
\newcommand{\thmref}[1]{Theorem~\ref{#1}\xspace}
\newcommand{\lemref}[1]{Lemma~\ref{#1}\xspace}
\newcommand{\propref}[1]{Proposition~\ref{#1}\xspace}

\newcommand{\comment}[1]{}

\newcommand{\R}{{\mathbb{R}}}
\newcommand{\E}{{\mathbf{E}}}
\newcommand{\Var}{{\mathbf{Var}}}
\newcommand{\Cov}{{\mathbf{Cov}}}
\newcommand{\Prob}{{\mathbf{Pr}}}
\newcommand{\Bin}{{\textup{Bin}}}
\newcommand{\Text}[1]{\enskip\text{#1}\enskip}
\newcommand{\cei}[1]{{\left\lceil{#1}\right\rceil}}

\newcommand{\mc}[1]{\mathcal #1}

\newcommand{\G}[2]{\I G_{#1,#2}}
\newcommand{\res}[1]{\I Z/#1\I Z}
\newcommand{\hide}[1]{}
\newcommand{\I}[1]{{\mathbb #1}}
\newcommand{\Ov}[1]{\overline{#1}}
\newcommand{\eps}{\varepsilon}
\newcommand{\aut}{\mathrm{Aut}}
\newcommand{\shatten}[3]{\|#2-#1\|_{C_{#3}}}
\newcommand{\rt}{\right}
\newcommand{\lt}{\left}
\newcommand{\e}{\epsilon}

\title{How unproportional must a graph be?}

\author{Humberto Naves\fnref{fn1}}
\ead{hnaves@ima.umn.edu}
\address{Institute for Mathematics and its Applications,
	University of Minnesota, Minneapolis, MN 55455, USA.}
\author{Oleg Pikhurko\fnref{fn2}}
\ead{O.Pikhurko@warwick.ac.uk}
\address{Mathematics Institute and DIMAP,
		University of Warwick, Coventry CV4 7AL, UK.}
	\author{Alex Scott}
	\ead{scott@maths.ox.ac.uk}
	\address{Mathematical Institute, University of Oxford, Andrew Wiles Building, Radcliffe Observatory
			Quarter, Woodstock Road, Oxford OX2 6GG, UK.}

\fntext[fn1]{Supported in part by the Institute for
	Mathematics and its Applications with funds provided by
	the National Science Foundation.}

\fntext[fn2]{Supported in part by ERC grant~306493 and EPSRC grant~EP/K012045/1.}

\begin{document}

\begin{abstract}
	
	Let $u_k(G,p)$ be the maximum over all $k$-vertex graphs $F$
	of by how much the number of induced copies of $F$ in $G$ differs
	from its expectation in the binomial random graph with
	the same number of vertices as $G$ and with edge probability $p$.
	This may be viewed as a measure of how close $G$ is to being
	$p$-quasirandom. For a positive integer $n$ and $0<p<1$, let
	$D(n,p)$ be the distance from $p\binom{n}{2}$ to the nearest
	integer. Our main result is that, for fixed $k\ge 4$
	and for $n$ large, the minimum of $u_k(G,p)$ over $n$-vertex
	graphs has order of magnitude $\Theta\big(\max\{D(n,p), p(1-p)\}
	n^{k-2}\big)$ provided that $p(1-p)n^{1/2} \to \infty$.
	
\end{abstract}

\maketitle

\section{Introduction}
\label{sec:intro}

An important result of Erd\H{o}s and Spencer~\cite{ErdosSpencer71}
states that every graph $G$ of order $n$ contains a set $S\subseteq
V(G)$ such that $e(G[S])$, the number of edges in the subgraph induced by
$S$, differs from $\frac12{|S|\choose 2}$ by at least $\Omega(n^{3/2})$; an
earlier observation of Erd\H{o}s~\cite{Erdos63a} shows that this lower
bound is tight up to the constant.
More generally, it was shown in \cite{ErdosGoldbergPachSpencer88}
that for graphs with density $p\in(\frac{2}{n-1},1-\frac{2}{n-1})$,
there is some subset where the number of edges differs from expectation
by at least $c\sqrt{p(1-p)}n^{3/2}$
(see \cite{BollobasScott06,BollobasScott11,BollobasScott15} for further results and discussion).

When $p$ is constant, the above results can be equivalently reformulated
in the language of graph limits as that the smallest cut-distance from
the constant-$p$ graphon to an order-$n$ graph $G$ is $\Theta(n^{-1/2})$.
Instead of defining all terms here
(which can be found in Lov\'asz' book~\cite{Lovasz:lngl}), we observe that
the cut-distance in this special case is equal, within some multiplicative
constant, to the maximum over $S\subseteq V(G)$ of $\frac1{n^2}\,
\Big|2e(G[S])-p|S|^2\Big|$.

There are other measures of how close a graph $G$ is to the constant-$p$
graphon, which means measuring how close $G$ is to being $p$-quasirandom.
Here we consider two possibilities, subgraph statistics and graph
norms, as follows. 

For graphs $G$ and $H$, we denote by $N(H, G)$ the number of induced
subgraphs of $G$ that are isomorphic to $H$. For example, if $v(H)=k\le n$,
then the expected number of $H$-subgraphs in the binomial random
graph $\G np$ (where each pair on the vertex set $[n]:=\{1,\dots,n\}$ is independently included as an edge with probability $p$) is
\[
  \E[N(H,\G np )]= \frac{n(n-1)\dots (n-k+1)}{|\aut(H)|}\,
  p^{e(H)}(1-p)^{{k\choose2}-e(H)},
\]
where $\aut(H)$ is the group of automorphisms of $H$.

Let $k\ge 2$ be a fixed integer parameter. For any graph $G$ on $n$
vertices and a real $0<p<1$, let
\begin{equation}
  \label{eq:u_k_g}
  u_k(G,p) := \max \Big\{\,\big|\,N(F, G) - \E[N(F, \G np )]\,\big|\, : v(F) = k\Big\},
\end{equation}
where the maximum is taken over all (non-isomorphic) graphs $F$ on $k$
vertices. The quantity $u_k(G,p)$ measures how far the graph
$G$ is away from the random graph $\G np $ in terms of $k$-vertex induced
subgraph counts. For example, $u_k(G,p)/n^k$ is within a constant factor (that depends on $k$
only) from the total variational distance between $\G kp$ and a random
$k$-vertex subgraph of $G$.

We are interested in estimating
\begin{equation}
  \label{eq:u_k_n}
  u_k(n,p) := \min\{u_k(G,p): v(G)=n\},
\end{equation}
 the minimum value of $u_k(G,p)$ that a graph $G$ of order $n$ can have. 
Informally speaking, we ask how $p$-quasirandom a graph of order $n$ can be.

Clearly, $u_2(n,p) < 1$ and $u_2(n,p) = 0$ if $p\binom{n}{2}$ is integer.
In fact, if we denote by $D(n,p)$ the distance from $p\binom{n}{2}$ to
the nearest integer, then $u_2(n,p) = D(n,p)$. The problem of constructing pairs
$(F,p)$ with $u_3(F,p)=0$ (such graphs $F$ were called \emph{$p$-proportional})
received some attention because the Central Limit Theorem fails for the random
variable $N(F,\G np )$ for such $F$, see~\cite{BarbourKaronskiRucinski89,Janson90flt,JansonNowicki91}. Apart from sporadic examples,
infinitely many such pairs were constructed by Janson and
Kratochvil~\cite{JansonKratochvil91} for $p=1/2$ and by Janson and
Spencer~\cite{JansonSpencer92} for every fixed
rational $p$; see K\"arrman~\cite{Karrman93} for a different proof
of the last result.

The main contribution of this paper is the following.
\begin{theorem}

  \label{thm:main}
  (a)  Let $k\ge 3$ be fixed and $p=p(n)\in(0,1)$ with $\frac{1}{p(1-p)} =
  o(n^{1/2})$. Then
  \[
    u_k(n,p) = O\big(\max\{D(n,p),\, p(1-p)\}n^{k-2}\big).
  \]

  (b)  Let $k\ge 4$ be fixed and $p=p(n)\in(0,1)$. Then
  \[
   u_k(n,p) = \Omega\big(\max\{D(n,p),\, p(1-p)\}n^{k-2}\big).
  \]

\end{theorem}

Note that the existence of proportional graphs shows that the lower bound of Theorem~\ref{thm:main} does  not
extend in general to $k=3$.

Another measure of graph similarity is the $2k$-th Shatten norm $\shatten pG{2k}$. Lemma~8.12
in~\cite{Lovasz:lngl} shows that the $4$-th Shatten norm defines the same
topology as the cut-norm. Again, we define it only for the special case when
we want to measure how $p$-quasirandom an $n$-vertex graph  $G$ is, where we allow loops. Here, we take the (normalised) $\ell_{2k}$-norm of the eigenvalues $\lambda_1,\dots,\lambda_n$ of
$M=A-pJ$, where $A$ is the adjacency matrix of $G$ and $J$ is the all-$1$
matrix:
$$
\shatten pG{2k}:=\frac{(\lambda_1^{2k}+\dots+\lambda_n^{2k})^{1/2k}}n.
$$ 
 We remark that when $G$ has a loop, the corresponding diagonal
entry in the matrix $A$ is~1. An equivalent and more combinatorial
definition of the $2k$-th Shatten norm is to take $\shatten pG{2k}=
t(C_{2k},M)^{1/2k}$, where $C_{2k}$ is the $2k$-cycle and $t(F,M)$
denotes the \emph{homomorphism density} of a graph $F$, which is the
expected value of $\prod_{ij\in E(F)} M_{f(i),f(j)}$, where
$f:V(F)\to [n]$ is a uniformly chosen random function, see
\cite[Chapter 5]{Lovasz:lngl}. In other words,
\begin{equation}
  \label{eq:shatten}
  \shatten pG{2k}= \lt(n^{-2k}\sum_{f:\res{2k}\to [n]}\
    \prod_{i\in\res{2k}} (A_{f(i),f(i+1)}-p)\rt)^{1/2k},
\end{equation}
where the sum is over all $n^{2k}$ maps
$f:\res{2k}\to [n]$, from the integer residues modulo $2k$ to~$\{1,\dots,n\}$.

We can show the following result.

\begin{theorem}

  \label{th:shatten} 
  Let $k\ge 2$ be a fixed integer. The minimum of $\shatten pG{2k}$
  over all $n$-vertex graphs $G$ with loops allowed is
  \[
    \Theta\lt(\min\lt\{p(1-p),\, p^{1/2}(1-p)^{1/2} n^{-(k-1)/2k}\rt\}\rt).
  \]

\end{theorem}

Hatami~\cite{Hatami10} studied which graphs other than even cycles produce
a norm when we use the appropriate analogue of \eqref{eq:shatten}. He showed,
among other things, that complete bipartite graphs with both parts of even
size do. We also prove a version of Theorem~\ref{th:shatten} for this
norm, see Theorem~\ref{th:2k2m} of \secref{sec:shatten}.

The rest of this paper is organised as follows. In \secref{sec:lower_bound}
we prove the lower bound from \thmref{thm:main}. In \secref{sec:upper_bound}
we prove the upper bound. We consider graph norms in
\secref{sec:shatten}, in particular proving Theorem~\ref{th:shatten} there.
The final section contains some open questions and concluding remarks.
Throughout the paper, we adopt the convention that $k$ is a fixed
constant and all asymptotic notation symbols ($\Omega$, $O$, $o$ and $\Theta$)
are with respect to the variable $n$. To simplify the presentation, we often omit
floor and ceiling signs whenever these are not crucial and make no attempts
to optimise the absolute constants involved.

\section{Lower bound for $u_k(n,p)$ in the range $k \ge 4$}
\label{sec:lower_bound}

The goal of this section is to prove that $u_k(n,p) =
\Omega\big(\max\{D(n,p),\, p(1-p)\}n^{k-2}\big)$. 
More precisely, we will show that there exists a constant $\eps=
\eps(k)>0$ such that $u_k(G,p) \ge \eps\max\{D(n,p),\,p(1-p)\} n^{k-2}$,
for all graphs $G$ on $n\ge k$ vertices and for all $0<p<1$. The
following lemma shows that it is enough to prove the lower bound for
$k=4$ only.

\begin{lemma}

  \label{lem:step_up}
  For every $k\ge 2$ there is $c_k>0$ such that
  $u_{k+1}(G,p)\ge c_kn \cdot u_k(G,p)$ for every graph $G$ of order
  $n\ge k+1$ and for all $0<p<1$.

\end{lemma}
\begin{proof}
Define 
 $$u_F(G,p):=\big|N(F,G)-\E[N(F,\G np )]\big|.
 $$ Take a graph $F$ of
order $k$ with $u_F(G,p)= u_{k}(G,p)$. Let $f(G)$ be the number of
pairs $(A,x)$ where a $k$-set $A$ induces $F$ in $G$ and $x\in V(G)\setminus A$.
Then $f(G)= (n-k)N(F,G)$ and $\E[f(\G np )]=(n-k)\E[N(F,\G np )]$;
thus these two parameters differ (in absolute value) by exactly
$(n-k)u_k(G,p)$. On the other hand, $f(G)$ can be written as
$\sum_{F'} N(F,F')N(F',G)$ where the sum is over non-isomorphic $(k+1)$-vertex
graphs $F'$. The expectation of $f(\G np )$ obeys the same
linear identity: 
\[
 \E[f(\G np )]=\sum_{v(F')=k+1} N(F,F')\,\E[N(F',\G np)].
\]
 We conclude that
 \begin{eqnarray*}
 \frac{n}{k+1}\, u_k(G,p)&\le& (n-k)\,u_k(G,p)\ =\ \big|\,f(G)-\E[f(\G np )]\,\big|\\ 
 &\le& \sum_{v(F')=k+1} N(F,F')\,
  u_{F'}(G,p)\ \le\ 2^{{k+1\choose 2}}\cdot (k+1)\cdot u_{k+1}(G,p).
 \end{eqnarray*}%
 Thus the lemma holds with $c_k=2^{-{k+1\choose 2}}\, (k+1)^{-2}$.
\end{proof}

In the next lemma we prove one of the bounds for $u_4(n,p)$. We remark
that it was implicitly proven in \cite[Proposition~3.7]{JansonKratochvil91}.

\begin{lemma}

  \label{lem:k_equals_4}
  There exists an absolute constant $\eps>0$ such that,
  for every $0<p<1$ and for all graphs $G$ on $n\ge 4$
  vertices, the inequality $u_4(G,p) \ge \eps p(1-p) n^2$ holds.

\end{lemma}
\begin{proof}
Let $\eps>0$ be a sufficiently small constant.
Suppose that there
is a graph $G$ of order $n\ge 4$ satisfying $u_4(G,p)<\eps p(1-p) n^2$.
By applying \lemref{lem:step_up} twice, we conclude that $u_2(G,p)
<\eps_1 p (1-p)$, where we set $\eps_1:=\eps/(c_2c_3)$ with the constants $c_i$ given by the lemma.
This implies that
 \begin{eqnarray}
  \bigg| e(G)^2 - \E\lt[e(\G np )\rt]^2\bigg|
  &\le&  \big| e(G) - \E\lt[e(\G np )\rt]\big|\cdot \lt(2p\binom{n}{2}+
     \eps_1 p(1-p)\rt)\nonumber\\
  &<& \eps_1 p (1-p)\cdot 3p\binom{n}{2}\ =\ 3\eps_1 p^2(1-p)\,\binom{n}{2}.\label{eq:eG2}
\end{eqnarray}%
\comment{OP: Indeed, we need to show that $ \frac{\eps p(1-p)}{c_2c_3}< p{n\choose 2}$, which is equivalent to $\eps (1-p)< {n\choose 2}c_2c_3$, which holds if eg $\eps<6c_2c_3$}

For every graph $G$, we can write $e(G)^2$ as
\begin{equation}
  \label{eqn:square_edges}
  e(G)^2 = \sum_{2\le v(F)\le 4} \alpha_F N(F,G),
\end{equation}
where  $F$ in the summation ranges over non-isomorphic graphs satisfying $2\le v(F)\le 4$,
and $\alpha_F\ge 0$ is a constant depending on $F$ only. Indeed, split ordered pairs $(e,e')\in E(G)^2$ according to the isomorphism type $F$ of $G[e\cup e']$. The number $\alpha_F$ of times that a given  $F$-subgraph  in $G$ is counted equals the number of ways to pick an ordered pair of edges from $E(F)$ whose union is the whole vertex set~$V(F)$. For example, if $F$
is an edge then $\alpha_F=1$, and if $v(F)=4$ then $\alpha_F$ is the
number of ordered pairs of disjoint edges in $F$.

Since $\E\lt[e(\G np )^2\rt] - \E\lt[e(\G np )\rt]^2 = \Var[e(\G np )]=
p(1-p){n\choose 2}$ is the variance of $e(\G np)$, we have by~\eqref{eq:eG2} and the Triangle Inequality that
\begin{equation}
  \label{eqn:lb_sq_edges}
  \bigg|e(G)^2 - \E\lt[e(\G np )^2\rt]\bigg|>p(1-p){n\choose 2}-3\eps_1 p^2(1-p)\,\binom{n}{2}>
\frac{p(1-p)}{2}\binom{n}{2}.
\end{equation}%
\comment{OP: here we need $\frac{3\eps p^2(1-p)}{c_2c_3}\binom{n}{2}<\frac{p(1-p)}{2}\binom{n}{2}$ for which
$3\eps<c_2c_3/2$ suffices}
Moreover, the identity \eqref{eqn:square_edges} implies that
$\E\lt[e(\G np )^2\rt] = \sum_{2\le v(F)\le4} \alpha_F\, \E[N(F,\G np )]$.
Thus, by~\eqref{eqn:lb_sq_edges},
\begin{eqnarray*}
 \sum_{k=2}^4\sum_{v(F)=k} \alpha_F \,u_k(G,p) &\ge&
 \sum_{k=2}^4 \sum_{v(F)=k} \alpha_F \Big|N(F, G)-\E[N(F,\G np )]\,\Big|\\
 &\ge& \bigg|\sum_{k=2}^4 \sum_{v(F)=k} \alpha_F \Big( N(F, G)-\E[N(F,\G np )]\,\Big)\bigg|\\
 &=& \bigg|e(G)^2 - \E\lt[e(\G np )^2\rt]\bigg|\ >\ \frac{p(1-p)}{2}\binom{n}{2}.
\end{eqnarray*}
Thus for some $k\in\{2,3,4\}$, we have $u_k(G,p)\ge\eps p(1-p)n^2$. \lemref{lem:step_up}
implies that $u_4(G,p)>\eps p(1-p) n^2$, contradicting our
assumption and proving the lemma.
\end{proof}
The previous two lemmas give that $u_k(n,p) = \Omega(p(1-p)
n^{k-2})$ for $k\ge 4$. Thus, in order to finish the proof of the lower bound,
we need to show that $u_k(n,p) = \Omega(D(n,p)n^{k-2})$.
The latter bound is a consequence of $u_2(n,p) = D(n,p)$
together with \lemref{lem:step_up}, thereby concluding the
proof of Theorem~\ref{thm:main}(b).

\section{Upper bound for $k\ge 3$}
\label{sec:upper_bound}

In this section, we prove that $u_k(n,p) = O(\max\{D(n,p), p(1-p)\}
n^{k-2})$ for fixed $k\ge 3$ and for all $p=p(n)$ such that $\frac{1}{p(1-p)}=
o(n^{1/2})$. We can assume, without loss of
generality, that $p \le \frac12$. Indeed, if $\Ov{G}$ denotes the
complement of $G$ then $u_k(G,p) = u_k(\Ov{G}, 1-p)$, which implies
that $u_k(n,p)=u_k(n,1-p)$. Thus our assumption can be made because
the bound $O(\max\{D(n,p), p(1-p)\}n^{k-2})$ is symmetric with
respect to $p$ and $1-p$. (Recall that $D(n,p) = u_2(n,p) = u_2(n,1-p)
= D(n,1-p)$.)
In addition, note that in the range $p\le \frac12$, it suffices
to show that $u_k(n,p)=O(\max\{D(n,p), p\} n^{k-2})$.

To prove the upper bound, we borrow some definitions, results, and
proof ideas from~\cite{JansonSpencer92}. Following their notation,
one can count the number of induced subgraphs of $G$ that are isomorphic
to $H$ using the following identity
\begin{equation}
  \label{eqn:num_iso}
  N(H,G) = \sum_{H'} \prod_{e\in E(H')} I_G(e)
    \prod_{e\in E(\overline{H'})} (1-I_G(e)),
\end{equation}
where we sum over all $H'$ isomorphic to $H$ with $V(H') \subseteq V(G)$,
$I_G(e)$ is the indicator function that $e$ is an edge in $G$ and $\overline{H'}$
denotes the complement of the graph $H'$. Observe that the range of $H'$ taken in the outermost sum in~\eqref{eqn:num_iso} depends on $V(G)$ but not on $E(G)$; this will be useful when comparing $H$-counts in different graphs on the same vertex set. We define a related sum over the same range of
$H'$:
\begin{equation}
  \label{eqn:signed_sum}
  S(H,G) = S^{(p)}(H,G):= \sum_{H'} \prod_{e\in E(H')} (I_G(e) - p),
\end{equation}
where $p$ is as before. Rewriting~\eqref{eqn:num_iso} by replacing each factor $I_G(e)$ by
$(I_G(e) - p) + p$  and each factor $1-I_G(e)$ by $(1-p)-(I_G(e)-p)$  and expanding,
we obtain a linear combination of products $\prod_{e\in X} (I_G(e)-p)$, with each $X$ being some subset of unordered pairs of $V(G)$ involving at most $v(H)$ different vertices. All sets $X$ that are isomorphic to the same graph $F$ get the same coefficient, which we denote $a_{F,H}(n,p)$. The coefficient for $X=\emptyset$ (i.e.\ the constant term) is obtained by summing the same quantity $p^{e(H')}(1-p)^{e(\overline{H'})}$ over all summands $H'$; thus it is equal to the expected number of $H$-subgraphs in $\G np$. We separate this special term and re-write~\eqref{eqn:num_iso} as
\begin{equation}
  \label{eqn:id_iso_signed}
  N(H,G) = \E[N(H,\G np)]+ \sum_{F \in \mc{F}_k} a_{F,H}(n,p) S(F,G),
\end{equation}
where $k=v(H)$ and $\mc{F}_k$ denotes the family of all graphs $F$ without
isolated vertices satisfying $2\le v(F)\le k$. Also, note that $a_{F,H}(n,p)$ does not depend on $G$ and is bounded from above by $O(n^{v(H) - v(F)})$. In fact, one can show that $a_{F, H}(n,p) = O(p^{e(H)-\alpha}
n^{v(H) - v(F)})$, where $\alpha$ is the maximum number of edges that a common
subgraph of both $H$ and $F$ can have, but we will not need such an
estimate. 

Thus, in order to prove that there exists a graph $G$ on $n$ vertices such that
$u_k(G,p) = O(\max\{D(n,p), p\}n^{k-2})$, it suffices to show that there exists $G$
such that
\begin{equation}
  \label{eqn:small_signed}
  S(F,G) = \lt\{\begin{array}{ll}
    O(pn^{v(F)-2}), & \text{for all } F\in \mc{F}_k\setminus \{K_2\}, \\
    O(D(n,p)), & \text{if } F = K_2.
  \end{array}\rt.
\end{equation}
(Note that one cannot hope for $S(K_2,G)=O(p)$ in general; this is why we  need two terms in the asymptotic formula for $u_k(n,p)$.)
A natural candidate for $G$ in
\eqref{eqn:small_signed} is the random graph $G\sim \G np$. Unfortunately, $G$ does not work ``out of the box''; namely,~\eqref{eqn:small_signed} typically fails for $F\in\mc{F}_k$ with $v(F)\le 3$. However, by changing the adjacencies of carefully chosen pairs we can steer these parameters to have the desired order of magnitude.

The next lemma yields some bounds for $S(F,\G np )$.

\begin{lemma}

  \label{lem:small_signed_random}
  Let $G\sim \G np $. For all $F\in \mc{F}_k$, we have
  \[
    \E[S(F,G)]=0\quad\text{and}\quad
    \E[S(F,G)^2] \le p^{e(F)}n^{v(F)}.
  \]

\end{lemma}
\begin{proof}
By \eqref{eqn:signed_sum}, we have
\[
  \E[S(F,G)]=\sum_{F'} \enskip
  \E\lt[ \prod_{e\in E(F')} (I_{G}(e)-p)\rt],
\]
where the sum is over all $F'$ isomorphic to $F$ with $V(F') \subseteq
V(G)$. Each expectation on the right-hand side vanishes, by independence
and since $\E[I_{G}(e)]=p$. Thus $\E[S(F,G)]=0$. 

We similarly write
\[
  \E[S(F,G)^2]= \sum_{F', F''} \enskip\E\lt[ \prod_{e\in E(F')}
  (I_{G}(e)-p) \prod_{e\in E(F'')}
  (I_{G}(e)-p)\rt].
\]
where the sum is over all pairs $(F',F'')$ of graphs isomorphic to $F$ with
$V(F') \cup V(F'') \subseteq V(G)$. The expectation term in the above sum
vanishes when $F' \ne F''$ and it is equal to $(p-p^2)^{e(F)}\le p^{e(F)}$
when $F'=F''$. Since the number of possible choices for $F'$ is at most
$\binom{n}{f}\cdot f!\le n^{f}$, where $f=v(F)$, we conclude that
$\E[S(F,G)^2] \le p^{e(F)} n^{v(F)}$.
\end{proof}

Using Chebyschev's inequality (see, e.g., \cite[Theorem 4.1.1]{AlonSpencer16pm}),
we have that, for all $\lambda>0$,
\begin{equation}
  \label{eqn:cheby_1}
  \Prob\lt[\,\big|S(F,\G np )\big| \ge \lambda\cdot
  p^{e(F)/2} n^{v(F)/2}\,\rt] \le \lambda^{-2}.
\end{equation}
By the union bound combined with \eqref{eqn:cheby_1}, the random graph
$G\sim \G np $ satisfies the following property  with probability at least
$0.96$.

\vspace{15pt}
\noindent\textbf{Property A.}\enspace $|S(F,G)| \le 5|\mc{F}_k|^{1/2}
p^{e(F)/2}n^{v(F)/2}$ for all graphs $F\in \mc{F}_k$.
\vspace{15pt}

The inequality $p^{e(F)/2}n^{v(F)/2} \le p n^{v(F)-2}$ holds whenever
$v(F)\ge 4$. This is because every graph on $4$ or more vertices in
$\mc{F}_k$ has at least $2$ edges, since no vertex is isolated. In order to find a
graph satisfying the conditions expressed in \eqref{eqn:small_signed}, we just
need to adjust $G$ so that $S(K_2,G) = O(D(n,p))$ and $S(F,G)=O(pn^{v(F)-2})$
when $F\in \mc{F}_3\setminus \{K_2\}$.
The family $\mc{F}_3\setminus \{K_2\}$ consists of two graphs: the triangle $K_3$ and the 2-path $P_2$, the unique graph on three vertices having exactly two edges.
So, we just need to adjust $S(K_2,G)$, $S(K_3,G)$ and $S(P_2,G)$. 
This must be performed carefully, to prevent $S(F,G)$ from
changing too much for graphs $F\in \mc{F}_k$ with $v(F)\ge 4$.

Let us investigate what happens to $S(F,G)$ when we add or remove an edge.
Note that by ``edges'', we generally mean edges in the complete graph, i.e.,
all pairs $ij$ with $i,j\in V(G)$, and not only the pairs that happen to
be selected as the edges of $G$. For each pair $ij$ with $i,j\in V(G)$, let
\begin{equation}
  \label{eqn:signed_sum_ij}
  S_{ij}(F,G) := S(F,G \cup \{ij\}) - S(F,G\setminus \{ij\}),
\end{equation}
where $G\cup \{ij\}$ and $G\setminus \{ij\}$ represent the graphs obtained
from $G$ by adding and removing the edge $ij$, respectively. 
By expanding each of the two terms in \eqref{eqn:signed_sum_ij} using~\eqref{eqn:signed_sum}, we can write $S_{ij}(F,G)$ as the sum  of
 $\prod_{e\in E(F')} (I_{G\cup \{ij\}}(e)-p)-\prod_{e\in E(F')} (I_{G\setminus \{ij\}}(e)-p)$ over all $F$-subgraphs $F'$ inside~$V(G)$. If $E(F')$ does not contain $ij$, then both products are identical. Thus we have that
 \begin{equation}\label{eq:SijFG}
  S_{ij}(F,G)=\sum_{F'} \big((1-p)-(-p)\big)\prod_{e\in E(F') \setminus \{ij\}} (I_{G}(e)-p)=\sum_{F'} \prod_{e\in E(F') \setminus \{ij\}} (I_{G}(e)-p),
 \end{equation}%
 where we sum over all $F'$ isomorphic to $F$ with $V(F') \subseteq
V(G)$ and $ij \in E(F')$.

The next
lemma gives a bound for the expectation and the variance of
$S_{ij}(F,\G np )$.

\begin{lemma}

  \label{lem:small_signed_random_ij}
  Let $G\sim \G np $. For all $F\in \mc{F}_k$ with $v(F)\ge 3$ and all pairs
  $1\le i < j \le n$, we have
  \[
    \E[S_{ij}(F,G)]=0\quad\text{and}\quad\E[S_{ij}(F,G)^2]
    \le k^2 p^{e(F) - 1} n^{v(F)-2}.
  \]

\end{lemma}
\begin{proof} The proof is similar to that of Lemma~\ref{lem:small_signed_random}.

We have $\E[S_{ij}(F,G)] = 0$ by~\eqref{eq:SijFG}, the independence of the random variables $I_G(e)$ and the linearity of expectation. 

For the second part of the lemma, we
write
\[
  \E[S_{ij}(F,G)^2]= \sum_{F', F''} \enskip\E\lt[
  \prod_{e\in E(F')\setminus \{ij\}} (I_{G}(e)-p)
  \prod_{e\in E(F'') \setminus \{ij\}} (I_{G}(e)-p)\rt].
\]
where the sum is over all pairs $(F',F'')$ of graphs isomorphic to $F$ with
$V(F') \cup V(F'') \subseteq V(G)$ and $\{i,j\} \in E(F') \cap E(F'')$.
The expectation term in the above sum vanishes when $F' \ne F''$ and it
is upper bounded by $p^{e(F) - 1}$ when $F'=F''$. Since the number of possible
choices for $F'$ is at most $k^2n^{v(F)-2}$, we conclude
that $\E[S_{ij}(F,G)^2] \le k^2p^{e(F) - 1}n^{v(F)-2}$, as desired.
\end{proof}

Take a pair $ij$ of vertices. For $0\le s\le 2$, let $Z_s=Z_s(ij)$  denote the number
of vertices $z\in V(G)\setminus \{i,j\}$ such that exactly $s$ of the pairs
$iz$ and $jz$ belong to $E(G)$. Let us express 
 \begin{eqnarray*}
  Y_1\ =\ Y_1(ij)&:=&S_{ij}(P_2,G),\\
  Y_2\ =\ Y_2(ij)&:=&S_{ij}(K_3,G),
  \end{eqnarray*} 
  in terms of the random variables~$Z_0$ and~$Z_2$.
 When we compute $Y_1$ using~\eqref{eq:SijFG}, we have to sum over all $2$-paths containing the edge $ij$. Denoting the third vertex of the path by $z$, we get 
  $$
 Y_1=\sum_{z\in V\setminus\{i,j\}} (I_G(iz)+I_G(jz)-2p)=2(1-p) Z_2 + (1-2p)Z_1 -2pZ_0.
 $$  
 Using 
 that $\E[Z_0]=(1-p)^2(n-2)$ and $\E[Z_2]=p^2(n-2)$ (or that $\E[Y_1]=0$), we derive that
 \begin{eqnarray}
  Y_1
  &=& 2(1-p) Z_2 + (1-2p)(n-2-Z_0-Z_2) -2pZ_0\nonumber\\
  &=&(Z_2-\E[Z_2])-(Z_0-\E[Z_0]).\label{eq:Y1}
  \end{eqnarray}
   Likewise, we obtain
  \begin{eqnarray}
  Y_2
  &=&\sum_{z\in V\setminus\{i,j\}} (I_G(iz)-p)(I_G(jz)-p)
  \ =\ (1-p)^2 Z_2 -p(1-p)Z_1+p^2 Z_0\nonumber\\
  &=&(1-p)(Z_2-\E[Z_2])+p(Z_0-\E[Z_0]).\label{eq:Y2}
  \end{eqnarray}


The triple $(Z_0, Z_1, Z_2)$ has a
multinomial distribution for $G\sim \G np $. 
In the next lemma we show that for any fixed rectangle
$R\subseteq \R^2$ of positive area, there exists $\eta =\eta(R)>0$ such that
$\lt(\frac{Y_1}{\sqrt{pn}}, \frac{Y_2}{p\sqrt{n}}\rt)\in R$
with probability at least $\eta$. Recall that we have assumed that $p\le 1/2$ and $p^2n\to\infty$.

\begin{lemma}

  \label{lem:limit_distr}
  For fixed reals $\alpha_1<\alpha_2$ and $\beta_1<\beta_2$ there exists
  $\eta=\eta(\alpha_1, \alpha_2, \beta_1, \beta_2) > 0$ such that, for all large $n$, the probability of
  \begin{equation}
    \label{eqn:cond_vars}
    \alpha_1\le\frac{Y_1}{\sqrt{pn}}\le\alpha_2\quad\text{and}\quad
    \beta_1\le\frac{Y_2}{p\sqrt{n}}\le\beta_2
  \end{equation}
  is at least $\eta$.
\end{lemma}
\begin{proof}
\comment{Our former proof very sketchy...}%
Define 
 \begin{eqnarray*}
 c&:=&\frac12\, \min\{\,\alpha_2-\alpha_1,\,\beta_2-\beta_1\,\},\\
 C&:=& 2\max\big\{\,|\alpha_1|,|\alpha_2|,|\beta_1|,|\beta_2|\,\big\},\\
 \delta&:=& \frac{c}{8\pi}\, \mathrm{e}^{-2C^2}\ >\ 0.
 \end{eqnarray*}
 
Let us show that $\eta:=\delta^2$ works in the lemma. 
 Consider the following $2\times 2$-matrix and its inverse:
 $$
 A:=\left[\begin{array}{cc}-1 & \sqrt{p}\\
 \sqrt{p} & 1-p\end{array}\right]\quad\mbox{with}\quad A^{-1}=\left[\begin{array}{cc} -1+p& \sqrt{p}\\ \sqrt{p} & 1\end{array}\right].
 $$
  Note that each entry of $A$ and $A^{-1}$ has absolute value at most $1$, so the linear maps given by these matrices are 2-Lipschitz in the $\ell_1$-distance. Thus if we let $S=S(n)$ be the square of side length $c$ with centre $(\alpha_0,\beta_0)^T:=A^{-1}(\frac{\alpha_1+\alpha_2}2,\frac{\beta_1+\beta_2}2)^T$,
  then the image of $S$ under $A$ lies inside the rectangle $R:=[\alpha_1,\alpha_2]\times [\beta_1,\beta_2]$ while $S$ itself is a subset of $A^{-1}R\subseteq [-C,C]^2$. (Here $(\alpha,\beta)^T$ means the column vector with entries $(\alpha,\beta)$.)

The matrix $A$ was chosen to encode the linear relations~\eqref{eq:Y1} and~\eqref{eq:Y2} between $(Y_1,Y_2)$ and $(Z_0,Z_2)$, with an appropriate normalisation applied to each random variable. Specifically, it holds that
\begin{eqnarray}
   A\left(\frac{Z_0-\E[Z_0]}{\sqrt{pn}},\frac{Z_2-\E[Z_2]}{p \sqrt{n}}\right)^T
   &=& \lt(\frac{Y_1}{\sqrt{pn}},\, \frac{Y_2}{p\sqrt{n}}\rt)^T.
   \label{eq:ZY}
\end{eqnarray}

\comment{Calculation: LHS of last equality is
$$
\left(\frac{Z_2-\E[Z_2]}{p \sqrt{n}}\,\sqrt{p}-\frac{Z_0-\E[Z_0]}{\sqrt{pn}},\ \frac{(1-p)(Z_2-\E[Z_2])}{p \sqrt{n}}+\frac{Z_0-\E[Z_0]}{\sqrt{pn}}\,\sqrt{{p}}\right)
   $$}
 By~\eqref{eq:ZY} it is enough to show, that with probability at least $\eta$, we have
 \begin{eqnarray}
   \label{eqn:cond_Z0}
    \alpha_0-\frac c2\ \le \frac{Z_0-\E[Z_0]}{\sqrt{pn}}  &\le&
   \alpha_0+\frac c2,\\
    \label{eqn:cond_Z2}
    \beta_0-\frac c2 \ \le\  \frac{Z_2-\E[Z_2]}{p\sqrt{n}} &\le& \beta_0+\frac c2.
 \end{eqnarray}

A version of de Moivre-Laplace theorem (see e.g.~\cite[Theorem~1.6(i)]{Bollobas:rg}) states that, for any function $p=p(n)\in (0,1)$ with $p(1-p)n\to \infty$ and any reals $a<b$, if $X_n$ has the binomial distribution with parameters $(n,p)$, then
 \begin{equation}\label{eq:MoivreLaplace}
 \lim_{n\to\infty}\, \Prob\left[\, a\le \frac{X_n-np}{\sqrt{np(1-p)}}\le b\,\right] = \frac1{2\pi}\int_a^b \mathrm{e}^{-x^2/2}\mathrm{d}x.
 \end{equation}

Let $n$ be large. We begin by sampling $Z_2$. We know that $Z_2$ is distributed
according to the binomial distribution: $Z_2\sim \Bin(n-2,p^2)$. Its variance is $\Var[Z_2]=p^2(1-p^2)(n-2)$. Let $Z_2^*:=(Z_2-\E[Z_2])/\sqrt{\Var[Z_2]}$ be the normalised version of~$Z_2$. Note that the constraint \eqref{eqn:cond_Z2} is satisfied if and and only if $Z_2^*$ belongs to $\gamma_n\cdot [\beta_0-\frac c2,\beta_0+\frac c2]$, where $\gamma_n:=p\sqrt n/\sqrt{\Var[Z_2]}$ and $y\cdot X:=\{y\cdot x: x\in X\}$ denotes the dilation of a set $X$ by a scalar~$y$.
De Moivre-Laplace theorem~\eqref{eq:MoivreLaplace} applies to $Z_2$ since we assumed that $p^2n\to \infty$ and $p\le 1/2$. Using $p\le 1/2$ again, we have that $\gamma_n$ is between, for example, $1$ and $2$.
Note that the normal distribution assigns probability at least $2\delta$ to every interval of length $c$ inside $[-2C,2C]$ by the definition of~$\delta$. 

Let us show that the probability of~\eqref{eqn:cond_Z2} is at least $\delta$. If this is false, then by passing to a subsequence of counterexamples $n$ we can further assume that $\gamma_n$ and $\beta_0=\beta_0(n)$ converge  to some $\gamma$ and $\beta$ respectively (with $\gamma\in [1,2]$ and $|\beta|\le C -c/2$). Let $I=[a,b]$ be the interval with centre at $\frac{a+b}2=\gamma\beta$ such that de Moivre-Laplace theorem predicts the limiting probability $\frac32\, \delta$ for it. Its length $a-b$ is strictly smaller than $\gamma c$ because, as we have already observed, the probability that the normal variable hits $\gamma\cdot [\beta-\frac c2,\beta+\frac c2]$ is at least~$2\delta$. 
Thus, for all large $n$ from our subsequence, $I$ is a subset of $\gamma_n\cdot [\beta_0(n)-\frac {c}2,\,\beta_0(n)+\frac {c}2]$. However, our assumption states that each of the latter intervals is hit with probability less than $\delta$ by $Z_2^*$, contradicting de Moivre-Laplace theorem when applied to the constant interval~$I$.
\hide{A standard compactness argument based on the boundedness of $\beta_0=\beta_0(p)$ implies that the probability of \eqref{eqn:cond_Z2} is at least $\delta$ for all large~$n$.}

Let $\alpha\in\{0,\dots,n-2\}$ be such that $|\beta-\beta_0|\le c/2$, where we set $\beta:=(\alpha-(n-2)p^2)/(p\sqrt{n})$.
Let $X_\alpha$ be $Z_0$ conditioned on $Z_2=\alpha$.
The random variable $X_\alpha$ has the binomial distribution with parameters $(1-p^2)(n-2)
-\beta p \sqrt{n}$ and $\frac{(1-p)^2}{1-p^2}=\frac{1-p}{1+p}$. 
By our assumption $p^2 n\to\infty$, the term $\beta p \sqrt{n}=O(p\sqrt{n})$ is negligible when compared to $p^2n$. We have
 \begin{eqnarray*}
  \E[X_\alpha] &=& (1-p)^2(n-2) - \frac{1-p}{1+p}\cdot \beta p \sqrt{n},\\
  \Var[X_\alpha] &=& (1+o(1))\, \frac{1-p}{1+p} \cdot \frac{2p}{1+p}\cdot (1-p^2)n\ =\ (2+o(1)) \frac{p(1-p)^2n}{1+p}.
 \end{eqnarray*}
 We see that $\Var[X_\alpha]$ lies between, for example, $np/4$ and $4np$.
As before, a compactness argument based on de Moivre-Laplace theorem shows that the infimum over all intervals $I\subseteq [-2C,2C]$ of length $c/2$  of the probability that $(X_\alpha-\E[X_\alpha])/\sqrt{\Var[X_\alpha]}$ belongs to $I$ is at least $\delta$ for all large~$n$. 

We see that, when conditioned on any value $\alpha$ of $Z_2$ that satisfies~\eqref{eqn:cond_Z2},
the probability that~\eqref{eqn:cond_Z0} holds is at least~$\delta$. 
Therefore, the probability that~\eqref{eqn:cond_Z0} and~\eqref{eqn:cond_Z2} hold simultaneously is at least $\eta = \delta^2$, which concludes the proof.
\end{proof}

Next, we put a pair $e\subseteq V(G)$ in at most one of sets $E_1,\dots,E_5$ as follows:
\begin{align*}
E_1 &:= \{e: e\in E(G),\, \sqrt{pn}<Y_1(e) \Text{and} p\sqrt{n}<Y_2(e)\}, \\
E_2 &:=\{e: e\in E(G),\, \sqrt{pn}<Y_1(e) \Text{and} Y_2(e)<-p\sqrt{n}\}, \\
E_3 &:=\{e: e\in E(G),\, Y_1(e)<-\sqrt{pn} \Text{and}
p\sqrt{n}<Y_2(e)\}, \\
E_4 &:=\{e: e\in E(G),\, Y_1(e)<-\sqrt{pn} \Text{and}
Y_2(e)<-p\sqrt{n}\}, \\
E_5 &:=\{e: e\not\in E(G),\, |Y_1(e)|<0.1\sqrt{pn} \Text{and}
|Y_2(e)|<0.1p\sqrt{n}\}.
\end{align*}
 Also, let $E^{*}$ denote the set of pairs $ij$, where $i,j \in V(G)$ are distint vertices such that
 \begin{equation}
 \label{eqn:cheby_2}
 |S_{ij}(F, G)| > 4k\cdot \eps^{-1/2}|\mc{F}_k|^{1/2}
 p^{(e(F) - 1)/2}n^{v(F)/2-1}
 \end{equation}
 for at least one $F\in \mc{F}_k$. 

Informally speaking, the rest of the proof proceeds as follows. First, by using Lemma~\ref{lem:limit_distr} we show  that, with reasonably high probability, the set $E_i\setminus E^*$ is ``large'' for each $i\in [5]$. Then, by applying a simple greedy algorithm, Corollary~\ref{cor:large_matching} gives a bounded degree graph $H'$ consisting of $\Omega(n)$ edges from each $E_i\setminus E^*$. We will modify
the random graph $G$ to satisfy~\eqref{eqn:small_signed} by flipping some pairs, all restricted to~$H'$. First, by flipping the appropriate number of pairs inside either $E_1$ or $E_5$, we can make $|S(K_2,G)|$ to be equal to $D(n,p)$, the smallest possible value, thus
satisfying one of the constraints in~\eqref{eqn:small_signed}. Next, by adding an edge from $E_5$ to $E(G)$ and removing an edge
in $E_i$ from $E(G)$, we do not change $S(K_2,G)$ while we can steer each of $S(K_3,G)$ and $S(P_2,G)$ in the right direction by having the freedom to choose $i\in[4]$. The latter claim can be justified using the fact that all flipped pairs come from a bounded degree graph $H'$, so the updated values of $Y_1(e)$ and $Y_2(e)$ stay close to the initial values for every pair~$e\subseteq V(G)$. Furthermore, since $H'$ is disjoint from $E^*$, the effect on $S(F,G)$ of every $H'$-flip is small for each $F\in\mathcal{F}_k$. Thus we make~\eqref{eqn:small_signed} hold for $F\in\mathcal{F}_3$ without violating it for the graphs in $\mathcal{F}_k\setminus\mathcal{F}_3$.

Let us provide all the details. Let $\eps>0$ be sufficiently small, in particular so that $\eta=\eps$ satisfies Lemma~\ref{lem:limit_distr} for any choice of $\alpha_1<\alpha_2$ and $\beta_1<\beta_2$ from $\{\,\pm0.1,\,\pm1,\,\pm2\,\}$.

First, let us show that $|E_1|\ge \eps pn^2/4$ asymptotically almost surely. Recall that
$E_1$ consists of those pairs $e\subseteq V(G)$ for which
\begin{equation}
  \label{eq:class1}
  e \in E(G),\quad \sqrt{pn} < Y_1(e)\quad\text{and}\quad
  p\sqrt{n} < Y_2(e).
 \end{equation}
Let $I_1(e)$ be the indicator random variable for $E_1$.
For the random graph $G\sim \G np $, the first condition $e\in E(G)$ for
$e$ to be in $E_1$ is independent of the other two conditions. Thus, by the choice of $\eps$, we can
assume that $\E[I_1(e)] \ge \eps p$. We have $|E_1|=\sum_{e} I_1(e)$, hence
$\E[\,|E_1|\,] \ge \eps p \binom{n}{2}$. We re-write the variance of $|E_1|$
as the sum of pairwise covariances of its components: with $\Cov[X,Y]:=\E[XY]-\E[X]\,\E[Y]$ we have
\begin{equation}\label{eq:VarE1}
  \Var[\,|E_1|\,] = \sum_{e\cap e' =\emptyset} \Cov[I_1(e),I_1(e')]
  +\sum_{e\cap e'\ne \emptyset} \Cov[I_1(e), I_1(e')],
 \end{equation}

Take any pairs $e=xy$ and $e'=x'y'$ that have no common vertices. Let us show that $\Cov[I_1(e),I_1(e')] = o(p^2)$. Informally speaking, $I_1(e)$ can only influence
$I_1(e')$ through the four edges that connect $e$ to $e'$, while the probability that $Y_1$ or $Y_2$ is so close to the cut-off values in~\eqref{eq:class1} as to be affected by these four edges is $o(1)$ by de Moivre-Laplace theorem. A bit more formally, we first expose all edges between the set $A:=e\cup e'$ and its complement $V(G)\setminus A$, and compute the ``current'' values $Y_1'$ and $Y_2'$ on $e$ and $e'$ where, for example,  
 $$
 Y_1'(e):=\sum_{z\in V(G)\setminus A} (I_G(xz)+I_G(yz)-2p)
 $$ 
 takes into account those 2-paths on $V(G)$ that contain $e=xy$ as an edge but are vertex-disjoint from the other pair~$e'$. The values of $Y_1$ and $Y_2$ on $e$ and $e'$ can be computed from $Y_1'$ and $Y_2'$ by adding the contribution from the four edges connecting $e$ to $e'$. By~\eqref{eq:Y1} and~\eqref{eq:Y2}, each of these increments is at most $8$. If $Y_1'(e),Y_1'(e')\not\in \sqrt{pn}\pm8$ and $Y_2'(e),Y_2'(e')\not\in p\sqrt{n}\pm 8$, then the validity of the requirements on $Y_1$ and $Y_2$ in~\eqref{eq:class1} does not depend on the four edges between $e$ and $e'$;
thus the corresponding contribution to $\Cov[I_1(e),I_1(e')]$ is zero. The complementary event, that at least one of $Y_1'$ and $Y_2'$ is within additive constant 8 from the corresponding cut-off value, has probability $o(1)$ by an application of de Moivre-Laplace theorem. Furthermore,
the constraints $e,e'\in E(G)$ in~\eqref{eq:class1}, that are independent of everything else, contribute $O(p^2)$ to the covariance of $I_1(e)$ and $I_1(e')$. Thus indeed $\Cov[I_1(e),I_1(e')] = o(p^2)$. 

We see that the first sum in
\eqref{eq:VarE1} has $O(n^4)$ terms, each $o(p^2)$. Since the second sum
has $O(n^3)$ terms, each at most $p^2$, the variance of $|E_1|$ is
$o(n^4p^2)$. By Chebyschev's inequality,
\[
  \Prob[\,|E_1| < \eps p n^2 / 4\,]\ \le\ \Prob[\,|E_1-\E[E_1]| > \eps p n^2 / 5\,]\  =\ o(1),
\]
proving the required.

The argument above implies that asymptotically almost surely
$|E_i| \ge \eps p n^2/4$ for all $i=1,\ldots, 4$. Similarly, one
can show that $|E_5| \ge \eps n^2/4$ asymptotically almost
surely. (Note that $E_5$ might be much ``denser'' than the other sets because
we dropped the requirement $e\in E(G)$.) 
Finally, using the standard
Chernoff estimates one can show that asymptotically almost surely
$\Delta(G) \le 2np$ for $G\sim \G np$. In particular, the following
property is satisfied with probability at least $0.99$ when $n$ is
large.

\vspace{15pt}
\noindent \textbf{Property B.}\enspace $|E_i|\ge \eps p n^2/4$
for $i=1,\ldots, 4$. Moreover, $|E_5| \ge \eps n^2/4$
and $\Delta(G)\le 2p n$.
\vspace{15pt}

Next, we would like to show that the set $E^*$ that was defined by~\eqref{eqn:cheby_2} is small. 
Chebyschev's inequality together with
\lemref{lem:small_signed_random_ij} implies that $\Prob[ij\in E^{*}]
\le \eps/16$. Hence $\E[\,|E^{*}|\,] \le \eps n^2/32$. By
Markov's inequality, $\Prob[\,|E^{*}| > \eps n^2/8\,] < \frac{1}{4}$.
Similarly, $\Prob[\,|E^{*}\cap E(G)| > \eps p n^2/8\,] < \frac{1}{4}$.
Thus by the union bound, $G\sim \G np $ satisfies the following property with
probability at least $0.5$.

\vspace{15pt}
\noindent \textbf{Property C.}\enspace $E^{*}$ has size at most
$\eps n^2/8$. Moreover, $|E^{*}\cap E(G)| \le \eps p n^2/8$.
\vspace{15pt}

Also, we state and prove the following simple result that asserts the existence
of large matchings in relatively dense graphs.

\begin{proposition}

  \label{prop:large_matching}
  Let $H$ be a graph and let $\Delta:=\Delta(H)$. There exists a
  matching in $H$ of size at least $\frac{e(H)}{2\Delta}$. In
  particular, if $m < \Delta$ then $H$ contains a
  subgraph $H'$ with maximal degree $\Delta(H')\le m$
  and $e(H') \ge \frac{m}{4\Delta} e(H)$.

\end{proposition}
\begin{proof}
Let $M$ be a maximal matching in $H$, and assume $M$ has $k<
\frac{e(H)}{2\Delta}$ pairs. All the edges of $H$ have at least one
endpoint in $V(M)$. Hence
\[
  e(H) \le |V(M)|\cdot \Delta = 2k\cdot \Delta < e(H),
\]
a contradiction. We remark that the bound $\frac{e(H)}{2\Delta}$ is not tight
but it suffices for our purposes.

To construct $H'$, we start with the empty graph.
At each step of the construction, we apply the first assertion of the
proposition to the graph $H\setminus H'$, in order to obtain a matching
$M$ having exactly $\cei{\frac{e(H)}{4 \Delta}}$ edges. We then add all the
edges from $M$ to $H'$. We repeat this step exactly $m$ times.
Since we always have $e(H') \le m \cdot \cei{\frac{e(H)}{4\Delta}}
< \frac{e(H)}{2}$, and thus $e(H\setminus H') > \frac{e(H)}{2}$,
it is always possible to find such $M$, in all the steps of the process.
\end{proof}

An important corollary of \propref{prop:large_matching} is as follows.

\begin{corollary}

  \label{cor:large_matching}
  Let $C>0$ be fixed. If Properties \textbf{B} and \textbf{C}
  simultaneously hold for a graph $G$ and $n$ is sufficiently large,
 then there exists a graph $H'$ having at least $Cn$ edges from each
  $E_i\setminus E^*$, $i=1,\ldots, 5$, such that $\Delta(H')\le 320C/\eps$.

\end{corollary}
\begin{proof}
Because of Property \textbf{C}, we have $|E^*\cap E(G)| \le \eps p n^2/8$ and
$|E^{*}| \le \eps n^2/8$, which, together with Property
\textbf{B}, implies that $|E_i\setminus E^*|
\ge \eps p n^2/8$ for $i=1,\ldots, 4$, and $|E_5\setminus E^*| \ge
\eps n^2/8$. Let $H_i$ be the graph on $V(G)$ having edge set
$E_i \setminus E^*$. We have $\Delta(H_i) \le \Delta(G) \le 2np$ for
$i=1,\ldots,4$ and $\Delta(H_5) \le n$. Hence $\frac{e(H_i)}{\Delta(H_i)} \ge
\frac{\eps n}{16}$ for all $i=1,\ldots, 5$. By \propref{prop:large_matching}
applied with $m=64 C / \eps < \min\{\Delta(H_i) : i=1,\ldots, 5\}$, 
each $H_i$ contains a subgraph $H_i'$ having at least $\frac{m}4\cdot
\frac{e(H_i)}{\Delta(H_i)}\ge C n$ edges
such that $\Delta(H_i')\le m$. Let $H' = \bigcup_{i=1}^5 H_i'$.
Clearly $\Delta(H') \le 5 m = 320 C/\eps$ and $H'$ contains at least
$Cn$ edges from each $E_i\setminus E^*$, thereby proving the corollary.
\end{proof}

\begin{proof}[Proof of the upper bound in \thmref{thm:main}]
Given $p\in (0,1/2]$ and $k\ge 3$, choose small $\eps>0$ and then sufficiently
large $C$. Let $n\to\infty$. By the union bound, $G\sim \G np $ satisfies
Properties~\textbf{A}, \textbf{B} and \textbf{C} with probability at least
$0.4$. Hence there exists a graph $G$ on $n$ vertices satisfying the three
properties simultaneously. Fix such $G$.

From Corollary~\ref{cor:large_matching}, there exists a graph $H'$ having
at least $Cn$ edges from each $E_i\setminus E^*$, such that
$\Delta:=\Delta(H') \le 320C/\eps$. Let $E' =E(H')$.

In what follows, we change $E(G)$ on pairs, all of which will belong
to $E'$. Note that at any intermediate step, the effect of (for instance)
removing an edge $ij\in E'\cap E_1$ from $E(G)$ on $S(P_2,G)$ and $S(K_3,G)$
is not quite given by the initial values of $Y_1(ij)$ and $Y_2(ij)$, since
certain edges $iw$, $jw$ might have been changed.
But $E'$ was defined in such a way that there are most $2\Delta=o(\sqrt{pn})$
changed edges which affect either $Y_1$ or $Y_2$. So, the removal of
$ij\in E_1\setminus E^*$ from $E(G)$ at any intermediate stage, still decreases
$S(P_2,G)$ by an amount between $\sqrt{pn}-2\Delta$ and $4k
\eps^{-1/2}|\mc{F}_k|^{1/2}\sqrt{pn}+2\Delta<\eps^{-1} \sqrt{pn}$. 
Similarly, because $\Delta = o(p\sqrt{n})$, the same operation decreases
$S(K_3, G)$ by an amount between $p\sqrt{n} - 2\Delta$ and
$4k\eps^{-1/2}|\mc{F}_k|^{1/2} p\sqrt{n} + 2\Delta < \eps^{-1} p\sqrt{n}$.

By Property \textbf{A}, we know that 
 $$|S(K_2,G)|\ \le\  5|\mc{F}_k|^{1/2}p^{1/2}n\ =:\ \tau.
 $$ 
 If $S(K_2,G) \ge 1$, we can pick an $e\in E'
\setminus E_5$ and remove it from $G$. This has the effect of
reducing $S(K_2,G)$ by $1$. If $S(K_2,G) \le -1$, then we can pick an
$e\in E'\cap E_5$ and add it to~$G$. This new edge increases the
value of $S(K_2,G)$ by $1$. Iterate this process at most
 $\tau$ times to obtain a graph $G$ such that
$|S(K_2,G)| = D(n,p)$, always using a different edge $e$. This is
possible because there are at least $C n$ edges from $E'\cap E_i$,
for each $i$.

Since we have flipped at most $\tau$ edges, all belonging to $H'$,
and each flip changes $S(K_3,G)$ (reps.\ $S(P_2,G)$) by at most $\eps^{-1}p\sqrt{n}$ (resp.\
$\eps^{-1}\sqrt{pn}$) in absolute value,
the current graph satisfies $|S(K_3,G)| \le p S_0$ and $|S(P_2,G)| \le p^{1/2} S_0$,
where
\[
  S_0=5|\mc{F}_k|^{1/2}p^{1/2}n^{3/2}+\tau \cdot
  \eps^{-1}\sqrt n.
\]

Our next goal is to make both $|S(K_3,G)|$ and $|S(P_2,G)|$ small
without changing $S(K_2,G)$.  We repeat the following step $Cp^{1/2}n-\tau$ times.
Consider the current graph $G$. There are four cases depending on whether
each of $S(K_3,G)$ and $S(P_2,G)$ is positive or not. First suppose that
they are both positive. Pick previously unused edges $e\in E' \cap E_1$ and $e'\in E'\cap E_5$, and replace $e$ with $e'$ in $G$. This operation preserves
the value of $S(K_2,G)$, and has the effect of reducing both $S(K_3,G)$ and
$S(P_2,G)$. It reduces $S(K_3,G)$ by between $(1-0.1)p\sqrt{n}-4\Delta\ge 0.8p\sqrt{n}$ and $2\eps^{-1}p
\sqrt n<pn$. Thus if (initially) $S(K_3,G)\ge pn$, then this value
is lowered by at least $0.8p\sqrt{n}$. Regarding $S(P_2,G)$, the operation
reduces it by between $0.8\sqrt{pn}$ and $2\eps^{-1}\sqrt{pn}<pn$.
Likewise, if $S(K_3,G)<0$ and $S(P_2,G)>0$, we replace an
$e\in E'\cap E_2$ by an $e'\in E'\cap E_5$, and similarly
in the other two cases. We iterate this process, always using 
edges $e$ and $e'$ that have not been used before. This is possible since
$E'$ contains at least $C n$ edges from each $E_i$.  Also, once one of
$|S(K_3,G)|$ or $|S(P_2,G)|$ becomes less than $pn$, it stays so for the rest
of the process. Since $(Cp^{1/2} n-\tau)\cdot 0.8\sqrt{n}> S_0$, we have
that $\max\{|S(K_3,G)|, |S(P_2,G)|\}<pn$ at the end.

The iterative process might change the value of $S(F,G)$ for $F\in \mc{F}_k$
with at least $4$ vertices. Take any such $F$ and let $f=v(F)$. Initially, $|S(F,G)|$ was at most $5|\mc{F}_k|^{1/2}p^{e(F)/2}n^{f/2}$ by Property~\textbf{A}. If we add to it $Cp^{1/2}n$, an upper bound on the number of the changed edges, multiplied by 
$4k\eps^{-1/2}|\mc{F}_k|^{1/2}p^{(e(F) - 1)/2}n^{f/2-1}$, then this accounts for every copy of $F$ inside
the vertex set $V(G)$ except perhaps those that contain at least two of the changed edges. (This estimate used the fact that none of the changed edges is in $E^*$.) A pair of two disjoint changed edges is trivially in at most $f^4 n^{f-4}$ copies of $F$. It remains to consider the case when $xy$ and $xz$ are two changed intersecting edges. Note that there are at most $Cp^{1/2}n\cdot 2\Delta$ choices of $(xy,xz)$. Consider a copy $F'$ of $F$ with vertex set $X\supseteq \{x,y,z\}$. If none of the pairs $e\subseteq X$ with $e\not\subseteq\{x,y,z\}$ is an element of $E(G)$ or a changed edge, then this $F'$ contributes at most $p$ in absolute value to the sum in~\eqref{eqn:signed_sum} that defines $S(F,G)$. (Indeed, as $F$ has at least 4 non-isolated vertices, at least one edge of $F'$ has to intersect $X\setminus\{x,y,z\}$; thus the $F'$-term in~\eqref{eqn:signed_sum} contains at least one factor~$-p$.) Otherwise, $X$ has to contain a changed edge or an edge from $E(G)$ that is not inside $\{x,y,z\}$. The number of such subgraphs 
for any given triple $\{x,y,z\}$ can be bounded by
 $$
 3(\Delta+2pn)f^4 n^{f-4} + ( Cp^{1/2}n + pn^2) f^5 n^{f-5}\le 2f^5pn^{f-3}.
 $$
 Putting all together we obtain that, at the end of the process,
 \begin{eqnarray*}
  |S(F,G)| &\le& 5|\mc{F}_k|^{1/2}p^{e(F)/2}n^{f/2}+Cp^{1/2}n \cdot
  4k\eps^{-1/2}|\mc{F}_k|^{1/2}p^{(e(F) - 1)/2}n^{f/2-1} \\
 &+& (Cp^{1/2}n)^2f^4n^{f-4}+Cp^{1/2}n\cdot 2\Delta\cdot (p\cdot f^3n^{f-3}+2f^5pn^{f-3}).
\end{eqnarray*}
\hide{
This is $O(p n^{v(F)-2})$ if $e(F) \ge 3$. In the remaining 
case $e(F) \le 2$ and $v(F)\ge 4$, we have only one graph, namely $F=2 K_2$ (two parallel
edges). It satisfies
\[
  S(K_2,G)^2 = p(1-p)\binom{n}{2} + (1-2p)S(K_2, G) + 2 S(P_2, G)+
  2S(2 K_2,G),
\]
therefore $|S(2K_2, G)| = O(pn^2)$.}%
This is $ O(pn^{f-2})$ since $F$ has $f\ge 4$ vertices and $e(F)\ge2$ edges.

We conclude that the final graph $G$ satisfies $S(F,G) = O(p n^{v(F)-2})$
for all $F\in \mc{F}_k\setminus \{K_2\}$ and $S(K_2,G)=O(D(n,p))$. That is, we satisfied~\eqref{eqn:small_signed}, which implies the required upper bound on $u_k(G,p)$.
\end{proof}

\section{Shatten norms and other related norms}
\label{sec:shatten}

Note that the graphs in this section are allowed to have loops. 
When we define
the complement $\Ov{G}$ of a graph $G$, loopless
vertices are mapped to loops and vice versa. 
For a graph $G$ on $[n]$ and a function $p=p(n)$, let $M=A-pJ$ denote the shifted adjacency matrix of $G$, that is,
\begin{equation}
\label{eqn:shifted_adj}
M_{ij}=\lt\{\begin{array}{ll} 1-p,& \text{if } ij\in E(G),\\
-p,& \text{otherwise},\end{array}\rt.\qquad 1\le i,j\le n.
\end{equation}
In order to make some forthcoming formulas shorter, we define $\e(G):=\sum_{i=1}^n \sum_{j=1}^n A_{ij}$. In other words, $\e(G)$ is the number of loops plus twice the number of non-loop edges in $G$. For example, $\e(G)+\e(\Ov{G})=n^2$.

Let us prove Theorem~\ref{th:shatten}

\begin{proof}[Proof of Theorem~\ref{th:shatten}{}]

Let $s=2k$  and let $G$ be a graph (possibly with loops) on $[n]$, where $n\to\infty$.
Without loss of
generality we may assume that $p \le \frac12$. This is because
$\shatten pG{s}^{s} = \shatten {(1-p)}{\Ov{G}}{s}^{s}$ and the
expression in the statement we have to prove is symmetric
with respect to $p$ and $1-p$.

The matrix $M$ in~\eqref{eqn:shifted_adj} is a symmetric real matrix so it has real eigenvalues 
$\lambda_1\ge\dots\ge\lambda_n$. For an even integer $s\ge 4$, we have 
\[
  n^s\,\shatten pG{s}^{s}=\sum_{i=1}^n \lambda_i^s=\textup{tr}(M^s)=\sum_{i=1}^n (M^s)_{ii},
\]
 where $\textup{tr}$ denotes the trace of a matrix.

From now on we split the analysis of the lower bound for
$\shatten pG{s}^{s}$ into two cases. 

In the first case, we assume that
$\e(G) \ge \frac p2\,n^2$. This (together with $p\le \frac12$)
implies that
\begin{equation}\label{eq:SumMijSquare}
\sum_{i=1}^n \lambda_i^2 =\sum_{i,j=1}^n M_{ij}^2=(1-p)^2\e(G)+p^2\e(\Ov{G})\ge
 \left((1-p)^2\,\frac p2+p^2\left(1-\frac p2\right)\right) \, n^{2}
=\frac p2\,n^2.
 \end{equation}
By the inequality between the arithmetic and $k$-th power means for
$k\ge 2$ applied to non-negative numbers $\lambda_1^2,\dots,\lambda_n^2$
(or just by the convexity of $x\mapsto x^k$ for $x\ge 0$), we conclude that
\[
  \lt(\frac{\lambda_1^{2k}+\dots+\lambda_n^{2k}}n\rt)^{1/k}\, \ge\,
  \frac{\lambda_1^2+\dots+\lambda_n^2}{n}\, \ge\, \frac{pn}2.
\]
Thus $n^{2k}\shatten Gp{2k}^{2k}=\sum_{i=1}^n\lambda_i^{2k}=
\Omega(p^k n^{k+1})$, giving the required lower bound in
the first case.

In the second case, we assume that $\e(G) < \frac p2\, n^2$. Since
$\lambda_n$ is the smallest eigenvalue of $M$, we have
$\lambda_n = \min \{\langle Mv, v\rangle : \|v\|_2 = 1\}$.
So if we choose $v=\lt(\frac{1}{\sqrt{n}},\ldots, \frac{1}{\sqrt{n}}\rt)
\in \R^n$, we obtain
\begin{equation}\label{eq:SumMij}
 \lambda_n \le \langle Mv, v\rangle = \frac{(1-p)\e(G) - p\e(\Ov{G})}{n}
  \le \left((1-p)\frac p2-p(1-\frac p2)\right)n= -\frac{pn}2.
\end{equation}
This implies that $\sum_{i=1}^n\lambda_i^{2k}\ge \lambda_n^{2k}
=\Omega(p^{2k} n^{2k})$, thereby proving the lower bound
in the second case.

On the other hand, for the upper bound we have two constructions.
Again we assume that $p \le \frac12$. The first construction is very
simple: the empty graph. If $G$ is empty, a straightforward computation
shows that $\shatten pG{2k} = p$,
and this proves the upper bound whenever $p \le  n^{-(k-1)/k}$.
For the second construction, we consider $G\sim \I G_{n,p}^\text{loop}$
to be a random graph with loops, where every possible pair or loop belongs to
$E(G)$ independently with probability $p$. Here we assume that $p>n^{-(k-1)/k}$.
Let $X=n^{2k}\shatten pG{2k}^{2k}$. By~\eqref{eq:shatten}, we have $X=\sum_{f:\res{2k}\to V(G)} X_f$,
where $X_f=\prod_{i\in \res{2k}} M_{f(i),f(i+1)}$ and $M=A-pJ$ is as
before. Then the expectation of $X_f$ is $0$ unless for every $i$
there is $j\not=i$ with $\{f(j),f(j+1)\}=\{f(i),f(i+1)\}$,
that is, every edge of $C_{2k}$ is glued with some other edge.
If $f$ is a map with $\E[X_f]\not=0$ then the image
under $f$ of the edge set of $C_{2k}$ is a connected multi-graph
where every edge (or loop) appears with even multiplicity,
so it contains at most $k+1$ vertices. Since the number of maps $f$
for which the image of $C_{2k}$ contains at most $e$
distinct edges (ignoring multiplicity) is $O(n^{e+1})$, we have
\[
  \E[X] = O\lt(\sum_{e=1}^{k} n^{e+1}p^e\rt) =
  O(n^{k+1}p^{k}),
\]
since $p > n^{-1}$. Now take an outcome $G$ such that the value of $X$
is at most its expected value. This finishes the proof of the theorem.
\end{proof}

A related result of Hatami~\cite{Hatami10} shows that a complete bipartite
graph $F=K_{2k,2m}$, with even part sizes $2k$ and $2m$, also gives
a norm by a version of \eqref{eq:shatten}. If $G$ is a graph on $[n]$, then this norm, for $G-p$, is
 $$
  \|G-p\|_F:=t(F,M)^{1/(2k+2m)}=n^{-1} X^{1/(2k+2m)},
  $$
   where $M$ is as in
\eqref{eqn:shifted_adj},
\[
 X:=\sum_{f:A\cup B\to V(G)}\ \prod_{a\in A}\ \prod_{b\in B}
M_{f(a),f(b)},
\]
and $A, B$ are fixed disjoint sets of sizes $2k$ and $2m$ respectively.

\begin{theorem}\label{th:2k2m}
Let $F=K_{2k,2m}$ with $1\le k\le m$. The minimum of $\|G-p\|_F$ over
$n$-vertex graphs $G$ with loops allowed is
\[
  \Theta\lt(\min\lt\{p^{4km}(1-p)^{4km},\, p^{2km}(1-p)^{2km}n^{-k}
  \rt\}^{1/(2m+2k)} \rt).
\]
\end{theorem}

\begin{proof}
For the same reasons stated in the beginning of the proof of
Theorem~\ref{th:shatten} we may assume, without loss of generality,
that $p \le \frac12$. We begin with the lower bound. We
rewrite $X$ by grouping all maps $f:A\cup B\to V(G)$ by the
restriction of $f$ to $A$. For every fixed $h:A\to V(G)$, we have
\[
  \sum_{g:B\to V(G)} \ \prod_{a\in A} \ \prod_{b\in B} M_{h(a),g(b)} = 
  \lt(\sum_{u\in V(G)}\  \prod_{a\in A}M_{h(a),u}\rt)^{2m}\ge 0.
\]
As in the proof of Theorem~\ref{th:shatten}, we divide the analysis into
two cases. 

In the first case, we assume that $\e(G)\ge \frac p2\, n^2$.
Let $\mc H$ be the set of all $h:A\to V(G)$ such that $h(2i-1)=h(2i)$
for all $i\in[k]$, where we assumed that $A=[2k]$. Note that $|\mc H| = n^k$.
If $h\in \mc H$ we have
\[
 \sum_{u\in V(G)}\ \prod_{a\in A} M_{h(a),u}=
 \sum_{u\in V(G)} \ \prod_{i\in [k]} M_{h(2i),u}^2.
\]
Thus by the convexity of $x\mapsto x^{2m}$ for $x\in \R$,
the convexity of $x\mapsto x^k$ for $x\ge 0$, and the calculation in~\eqref{eq:SumMijSquare}, we have that
\begin{align*}
X &= \sum_{h:A\to V(G)}
  \lt(\sum_{u\in V(G)} \prod_{a\in A} M_{h(a),u}\rt)^{2m}
  \ge \sum_{h\in \mc H}
  \lt(\sum_{u\in V(G)} \prod_{i\in [k]} M_{h(2i),u}^2\rt)^{2m}\\
  &\ge n^k\lt(\frac{1}{n^k}\sum_{h\in \mc H}\sum_{u\in V(G)}
   \prod_{i\in [k]} M_{h(2i),u}^2\rt)^{2m}=
  n^k\lt(\frac{1}{n^k}\sum_{u\in V(G)}
     \lt[\sum_{v\in V(G)} M_{v,u}^2\rt]^k\rt)^{2m} \\
  &\ge n^k\lt(\frac{1}{n^{k-1}} \lt[\frac{1}{n}
    \sum_{u\in V(G)} \sum_{v\in V(G)} M_{v,u}^2\rt]^k\rt)^{2m}
   \ge
    n^k\lt(\frac{1}{n^{k-1}} \lt[\frac{(1-p)^2\e(G)+p^2\e(\Ov{G})}{n}
       \rt]^k\rt)^{2m} \\
  &\ge n^k\lt(\frac{1}{n^{k-1}} \lt[\frac{pn}{2}
         \rt]^k\rt)^{2m}
         = \Omega\lt(p^{2km}n^{k+2m}\rt),
\end{align*}
which proves the lower bound in the first case.

In the second case, we assume that $\e(G)< \frac p2\, n^2$.
By the convexity of $x\mapsto x^{2m}$ and $x\mapsto x^{2k}$
for all $x\in \R$ and by the calculation in~\eqref{eq:SumMij}, we have that
\begin{align*}
X &= \sum_{h:A\to V(G)}
  \lt(\sum_{u\in V(G)} \prod_{a\in A} M_{h(a),u}\rt)^{2m}
  \ge n^{2k}\lt(\frac{1}{n^{2k}}\sum_{h:A\to V(G)}\sum_{u\in V(G)}
   \prod_{i\in [2k]} M_{h(i),u}\rt)^{2m}\\
  &=n^{2k}\lt(\frac{1}{n^{2k}}\sum_{u\in V(G)}
     \lt[\sum_{v\in V(G)} M_{v,u}\rt]^{2k}\rt)^{2m}
  \ge n^{2k}\lt(\frac{1}{n^{2k-1}} \lt[\frac{1}{n}
    \sum_{u\in V(G)} \sum_{v\in V(G)} M_{v,u}\rt]^{2k}\rt)^{2m}\\
   &= n^{2k}\lt(\frac{1}{n^{2k-1}} \lt[\frac{(1-p)\e(G)-p\e(\Ov{G})}{n}
       \rt]^{2k}\rt)^{2m}
  = \Omega\lt(p^{4km}n^{2k+2m}\rt),
\end{align*}
which proves the lower bound in the second case.

We turn to the upper bound. We need two constructions. The first
one is again the empty graph. If $G$ is empty then
\[
  \|G-p\|_F = p^{2km/(k+m)},
\]
and this proves the upper bound whenever $p \le n^{-1/(2m)}$.
The second construction is the random graph $G\sim\G np^\text{loop}$. Write
$X$ as the sum of $X_f$ over $f:A\cup B\to V(G)$. Each $f$ with $\E[X_f]\not=0$
maps $E(K_{2k,2m})$ into a connected multi-graph where every edge appears with
even multiplicity.
Consider the equivalence relation on $A\cup B$ given by one such $f$, where
two vertices in $A\cup B$ are equivalent if their images under $f$ coincide.
If non-trivial classes (i.e.,\ those containing more than one vertex) miss some $a\in A$ and some $b\in B$, then $\{f(a),f(b)\}$
is a singly-covered edge, a contradiction. Thus, non-trivial classes have to
cover at least one of $A$ or $B$ entirely, so the number of identifications is
at least $\min\{|A|,|B|\}/2=k$. It follows that the image of $F$ under $f$ has
at most $k+2m$ vertices. In fact, if the image of $F$ under $f$ contains
exactly $2k + 2m -t$ vertices (where $t\ge k$), the number of distinct edges
in the image of $F$ by $f$ is at least $4km - 2mt$. This is because every
``identification'' of vertices under the same equivalence class of $f$ can
``destroy'' at most $2m$ edges. Therefore
\[
  \E[X] = O\lt(\sum_{t=k}^{2k+2m-1} n^{2k+2m-t}p^{4km-2mt}\rt) = O(n^{k+2m}p^{2km}),
\]
since $p > n^{-1/(2m)}$. Now take an outcome $G$ such that the value of $X$
is at most its expected value. This finishes the proof of the theorem.
\end{proof}

\section{Concluding remarks and open questions}
\label{sec:concluding_remarks}

Observe that the result
of Chung, Graham, Wilson~\cite{ChungGrahamWilson89} implies that there cannot
be a graph $G$ with $t(K_2,A)=p$ and $t(C_4,A)=p^4$ where $0<p<1$
and $A$ is the adjacency matrix of $G$.
(Indeed, otherwise the uniform blow-ups of $G$ would form a quasirandom
sequence, which is a contradiction.) This argument does not work with
the subgraph count function $N(F,G)$. We do not know if the fact that $u_k(n,p)$ can be zero infinitely often for
$k=3$ (when $p$ is rational) but not for $k=4$ can directly be related to the fact
that quasirandomness is forced by $4$-vertex densities.

Let $\G nm$ be the random graph on $[n]$ with $m$ edges, where all
$\binom{\binom{n}{2}}{m}$ outcomes are equally likely. 
Janson~\cite{Janson94} completely classified the cases when the random
variable $N(F,\G nm)$ satisfies the Central Limit Theorem where
$n\to\infty$ and $m=\lfloor p{n\choose 2}\rfloor$. He showed that
the exceptional $F$ are precisely those graphs for
which $S^{(p)}(H,F)=0$ for every $H$ from the following set: connected
graphs with $5$ vertices and graphs without isolated vertices with $3$ or
$4$ vertices. It is an open question if at least one such pair $(F,p)$ with
$p\not=0,1$ exists, see, e.g., \cite[Page 65]{Janson94} and
\cite[Page 350]{Janson95}. Note that nothing is stipulated about $S^{(p)}(K_2,F)$.
In fact, it has to be non-zero e.g.\ by \thmref{thm:main}; moreover,
\cite[Theorem~4]{Janson94} shows that, for given $v(F)$ and $p$, the number of edges in such 
hypothetical $F$ is uniquely determined. \hide{
Using computer, K\"arrman~\cite{Karrman94} constructed a 64-vertex graph $F$
with $S^{(1/2)}(H,F)=0$ for $H=K_2$, $3$- and $4$-path, and $4$-cycle. As far
as we know, it is open if there are any further such examples $F$. 
}%
This indicates that the problem of understanding possible joint
behaviour of the $S$-statistics is difficult already for very small graphs. 

It would be interesting to extend \thmref{thm:main} to a wider range of $p$,
or to other structures such as, for example, $r$-uniform hypergraphs with respect to different notions of quasirandomness (see~\cite{ConlonHanPersonSchacht12,LenzMubayi15,Towsner17}). 
\hide{Unfortunately,
we could not determine the order of magnitude of the corresponding function
$u_k$ for these structures.}

\section*{Acknowledgements}

We thank the anonymous referees for the careful reading of the manuscript and helpful comments.

\section*{References}

\bibliography{final}

\end{document}